\documentclass[a4paper,11pt]{article}
\pdfoutput=1
\usepackage{hyperref}
\hypersetup{hypertexnames = false, bookmarksdepth = 2, bookmarksopen = true, colorlinks, linkcolor = black, citecolor = black, urlcolor = black, pdfstartview={XYZ null null 1}}

\usepackage{amsfonts}
\usepackage[fleqn, leqno]{amsmath}
\usepackage{amsthm}

\usepackage[capitalise,noabbrev]{cleveref}

\usepackage{booktabs}
\usepackage{colonequals}
\usepackage{enumitem}
\usepackage{fullpage}
\usepackage{mathtools}
\usepackage{parskip}
\usepackage{thmtools}
\usepackage{tikz-cd}
\usepackage{xparse}
\usepackage{xpatch}
\usepackage{xspace}

\usepackage[utf8]{inputenc}
\usepackage[T1]{fontenc}
\usepackage{libertine}
\usepackage[libertine]{newtxmath}
\usepackage[scaled=0.83]{beramono}
\usepackage{eucal}
\usepackage{microtype}
\usepackage{gitinfo2}

\usepackage{newtxtext,newtxmath}
\usepackage{dynkin-diagrams}

\usepackage[backend=biber, maxbibnames=99, datamodel=mrnumber, sortcites]{biblatex}

\DeclareFieldFormat{mrnumber}{%
  MR\addcolon\space
  \ifhyperref
    {\href{http://www.ams.org/mathscinet-getitem?mr=#1}{\nolinkurl{#1}}}
    {\nolinkurl{#1}}}

\renewbibmacro*{doi+eprint+url}{%
  \iftoggle{bbx:doi}
    {\printfield{doi}}
    {}%
  \newunit\newblock
  \printfield{mrnumber}%
  \newunit\newblock
  \iftoggle{bbx:eprint}
    {\usebibmacro{eprint}}
    {}%
  \newunit\newblock
  \iftoggle{bbx:url}
    {\usebibmacro{url+urldate}}
    {}}

\DeclareSourcemap{
  \maps[datatype=bibtex]{
    \map{
      \step[fieldset=issn, null]
    }
  }
}
\xpretobibmacro{bib+doi+url}
  {\iffieldundef{doi}
     {}
     {\clearfield{url}}}
  {}{}

\xpretobibmacro{cite+doi+url}
  {\iffieldundef{doi}
     {}
     {\clearfield{url}}}
  {}{}

\declaretheoremstyle[spaceabove = 3pt, spacebelow = 3pt, bodyfont = \itshape]{theorem}
\declaretheoremstyle[spaceabove = 3pt, spacebelow = 3pt]{remark}

\declaretheorem[style=theorem]{theorem}

\declaretheorem[style=theorem, sibling=theorem]{lemma}
\declaretheorem[style=theorem, sibling=theorem]{proposition}

\declaretheorem[style=remark, sibling=theorem]{remark}

\declaretheorem[style=theorem, numberwithin=section, title=Theorem]{alphatheorem}
\declaretheorem[style=theorem, sibling=alphatheorem, title=Conjecture]{alphaconjecture}

\crefname{alphatheorem}{Theorem}{Theorems}
\crefname{alphaconjecture}{Conjecture}{Conjectures}
\crefname{alphacorollary}{Corollary}{Corollaries}
\crefname{alphaproposition}{Proposition}{Propositions}

\crefformat{enumi}{#2\textup{(#1)}#3}

\mathchardef\mhyphen="2D
\newcommand\dash{\nobreakdash-\hspace{0pt}}

\let\oldbigwedge\bigwedge
\renewcommand\bigwedge{\oldbigwedge\nolimits}

\newcommand\field{\mathbf{k}}
\newcommand\GP{\mathrm{G}/\mathrm{P}}
\newcommand\tangent{\mathrm{T}}

\DeclareMathOperator\HH{H}
\DeclareMathOperator\OGr{OGr}

\DeclareMathOperator\SGr{SGr}

\title{Failure of Bott vanishing for (co)adjoint partial flag varieties}
\author{Pieter Belmans}

\begin{document}
\maketitle

\begin{abstract}
  Bott vanishing is a strong vanishing result for
  the cohomology of exterior powers of the cotangent bundle
  twisted by ample line bundles.
  Buch--Thomsen--Lauritzen--Mehta
  conjectured that partial flag varieties (which are not products of projective spaces)
  do not satisfy Bott vanishing,
  despite all their other nice properties.
  The cominuscule case is an easy application of the Borel--Weil--Bott theorem,
  following results of Snow.
  We show that
  the (co)adjoint partial flag varieties of all classical and exceptional Dynkin types
  also do not satisfy Bott vanishing,
  thus confirming the conjecture for this class of varieties.
\end{abstract}

\section{Introduction}
Let $X$ be a smooth projective variety over a base field~$\field$.
We say that $X$ satisfies \emph{Bott vanishing} if
\begin{equation}
  \label{equation:bott-vanishing}
  \HH^p(X,\Omega_X^q\otimes\mathcal{L})=0,
\end{equation}
for all integers~$q\geq 0$, $p\geq 1$,
and every ample line bundle~$\mathcal{L}$.
Working over~$\mathbb{C}$,
this sheaf cohomology group is the~$(q,p)$th Dolbeault cohomology of~$\mathcal{L}$.
The vanishing property \eqref{equation:bott-vanishing} is named after Bott,
who established it for~$\mathbb{P}^n$~\cite[Proposition~14.4]{MR0089473}.
It was subsequently generalised from~$\mathbb{P}^n$ to all smooth projective toric varieties,
where it is known as Bott--Danilov--Steenbrink vanishing,
see \cite[Theorem~2.4]{MR1916637} for a modern exposition.

In the \emph{opposite} direction,
i.e.,
concerning the \emph{failure} of Bott vanishing,
Buch--Thomsen--Lauritzen--Mehta
conjectured the following in \cite[Remark~2]{MR1464183},
for a different class of generalizations of~$\mathbb{P}^n$.
\begin{alphaconjecture}[Buch--Thomsen--Lauritzen--Mehta]
  \label{conjecture:btlm}
  Let~$X=\GP$ be a partial flag variety
  which is not isomorphic to a product of projective spaces.
  Then~$X$ does not satisfy Bott vanishing.
\end{alphaconjecture}
\vspace{-.5\baselineskip}
By applying the Borel--Weil--Bott theorem,
and the description of~$\Omega_X^q$ given by Snow \cite{MR0938025},
the failure of Bott vanishing
for (co)minuscule varieties
(which includes all type~A~Grassmannians~$\operatorname{Gr}(k,n)$ and quadric hypersurfaces)
is proven in~\cite[\S4.3]{MR1464183}.
This method however no longer applies in other settings,
because~$\Omega_X^1$ is no longer a completely reducible vector bundle.
Our main result is the following theorem,
confirming this conjectured failure of Bott vanishing for many more partial flag varieties.
\begin{alphatheorem}
  \label{theorem:main}
  Let~$X=\GP$ be an adjoint or coadjoint partial flag variety,
  over a field of characteristic~0,
  which is not isomorphic to~$\mathbb{P}^n$.
  Then~$X$ does not satisfy Bott vanishing.
\end{alphatheorem}
\vspace{-.5\baselineskip}
We prove \cref{theorem:main} for adjoint varieties in \cref{section:adjoint},
and for coadjoint varieties in \cref{section:coadjoint}.
The argument works on a type-per-type basis,
describing the smallest~$q\geq 1$
for which the vector
bundle~$\smash{\bigwedge^q\tangent_X\otimes\mathcal{O}_X(-1)\cong\Omega_X^{\dim
X-q}\otimes(\omega_X^\vee\otimes\mathcal{O}_X(-1))}$
admits a non-zero~$\HH^1$.
Fonarev establishes in \cite[Theorems~A and~B]{2506.09727},
the analogue of \cref{theorem:main}
for almost all generalised Grassmannians in types~B, C and~D which are not
(co)minuscule or (co)adjoint,
thus the only open cases of \cref{conjecture:btlm}
for maximal parabolic subgroups
are for certain cases in types~$\mathrm{E}_{6,7,8}$ and~$\mathrm{F}_4$,
and~$\mathrm{OGr}(n-1,2n+1)=\mathrm{B}_n/\mathrm{P}_{n-1}$.

\paragraph{Context}
If~$X$ is Fano,
then taking~$\mathcal{L}=\omega_X^\vee$
shows that~$\HH^1(X,\tangent_X)=\HH^1(X,\Omega_X^{\dim X-1}\otimes\omega_X^\vee)=0$,
so~$X$ is necessarily infinitesimally rigid.
This is a very strong constraint on Fano varieties satisfying Bott vanishing.
Recently,
the first non-toric rigid Fano varieties for which Bott vanishing holds were discovered:
\begin{itemize}[nosep]
  \item In dimension~2,
    the only non-toric infinitesimally rigid del Pezzo surface
    is the quintic del Pezzo surface.
    Totaro shows that it satisfies Bott vanishing~\cite[Theorem~2.1]{MR4082249}.
  \item In dimension~3,
    Totaro shows that
    there are precisely~19 non-toric~Fano~3-folds
    (amongst 105~deformation families)
    which satisfy Bott vanishing~\cite[Theorem~1.1]{MR4735363}.
  \item In arbitrary dimension,
    Torres shows that~$(\mathbb{P}^1)^n/\!/\mathrm{PGL}_2$
    satisfies Bott vanishing \cite{10.1307/mmj/20226298}.
\end{itemize}
Bott vanishing for non-rationally connected surfaces
is studied in \cite[\S3--\S6]{MR4735363}
and \cite{MR4363805}.

\medskip

\paragraph{Lifting Frobenius}
One important motivation for studying (the failure of) Bott vanishing,
is for its role as an obstruction to lifting varieties modulo~$p^2$
together with their Frobenius.
Deligne--Illusie \cite[Corollaire~2.11]{MR0894379}
showed that the existence of a lifting modulo~$p^2$
implies Kodaira--Akizuki--Nakano vanishing,
i.e., that~$\HH^p(X,\Omega_X^q\otimes\mathcal{L})=0$
whenever~$p+q>\dim X$,
and~$\mathcal{L}$ is an ample line bundle.
Buch--Thomsen--Lauritzen--Mehta \cite[Theorem~3]{MR1464183}
showed that including the Frobenius in the lifting
implies the (much stronger) Bott vanishing \eqref{equation:bott-vanishing}.
In \cite[\S4.2]{MR1464183} they gave examples of (co)minuscule partial flag varieties
where this obstruction does not vanish.
Using \cite[Corollary~1]{MR1464183}
this gives many partial flag varieties not admitting
a lifting modulo~$p^2$
together with their Frobenius, see \cite[Theorem~6]{MR1464183}.

Recently, Achinger--Witaszek--Zdanowicz
proved in \cite[Theorem~1]{MR4269423}
that \emph{all} partial flag varieties which are not isomorphic to
products of projective spaces
do not admit a lifting modulo~$p^2$
together with their Frobenius.
Their obstruction uses the geometry of rational curves
on Fano varieties with nef tangent bundle.
Combining \cite[Corollary~1]{MR1464183}
with \Cref{theorem:main} and \cite[Theorems~A and~B]{2506.09727}
gives as a corollary
a different proof of this non-liftability
for almost all partial flag varieties (only excluding some exceptional cases),
extending \cite[Theorem~6]{MR1464183}
using Bott non-vanishing as the obstruction.

\medskip

\paragraph{Hochschild--Kostant--Rosenberg decomposition}
In a different direction,
the Hochschild--Kostant--Rosenberg decomposition
expresses the Hochschild cohomology of~$X$
as
\begin{equation}
  \label{equation:hkr}
  \operatorname{HH}^i(X)
  \cong
  \bigoplus_{p+q=i}
  \HH^p(X,\bigwedge^q\tangent_X)
  \cong
  \bigoplus_{p+q=i}\HH^p(X,\Omega_X^{\dim X-q}\otimes\omega_X^\vee).
\end{equation}
Thus, Bott vanishing for Fano varieties gives a strong constraint on
the shape of the Hochschild--Kostant--Rosenberg decomposition:
only the global sections in \eqref{equation:hkr} are possibly non-zero.
This corollary of Bott vanishing
is established in \cite[Theorem~A]{MR4706032}
for (co)minuscule and (co)adjoint partial flag varieties.
In \cite{MR4578397} the decomposition \eqref{equation:hkr}
is computed for Fano 3-folds,
leading to a classification of Fano 3-folds having this vanishing property.

However, Belmans--Smirnov conjecture in \cite[Conjecture~F]{MR4706032}
that (co)minuscule and (co)adjoint partial flag varieties are
the \emph{only} partial flag varieties with this vanishing property for \eqref{equation:hkr}.
Hence, a positive answer to this conjecture implies
a positive answer to \cref{conjecture:btlm},
except (!) for (co)minuscule and (co)adjoint partial flag varieties.
Simultaneous with, and independent of this paper,
\cite[Conjecture~F]{MR4706032} has been established
by Fonarev for many generalised Grassmannians
in types~B, C, and D
which are not (co)minuscule and not (co)adjoint \cite[Theorems~A and~B]{2506.09727}.
This leaves open the failure of Bott vanishing for
(co)adjoint partial flag varieties,
and \cref{theorem:main} precisely establishes
the conjectured failure of Bott vanishing,
when it is not already provided by the non-vanishing of higher cohomology in \eqref{equation:hkr}.

For full flag varieties,
where~$\mathrm{P}=\mathrm{B}$ is a Borel subgroup,
some examples of the non-vanishing of higher cohomology in \eqref{equation:hkr}
(and thus failure of Bott vanishing)
are discussed in
\cite[Remark~2.2]{MR4057490}
and
\cite[Example~4.9(b)]{MR4923443}.

\paragraph{Conventions}
In what follows,
we work over an algebraically closed field~$\field$ of characteristic~0.
We denote by~$X$ the partial flag variety~$\GP$,
where~$\mathrm{G}$ is taken to be a simply connected and simple algebraic group
and~$\mathrm{P}\subseteq\mathrm{G}$ a parabolic subgroup.

\section{Adjoint case}
\label{section:adjoint}
First we consider the class of adjoint partial flag varieties.
These are partial flag varieties~$\GP$
where~$\mathrm{P}\subset\mathrm{G}$ is a specific parabolic subgroup.
An explicit geometric realization is obtained by taking
the~$\mathrm{G}$-dominant weight~$\lambda$ for the adjoint representation,
and considering the unique closed~$\mathrm{G}$-orbit
in~$\mathbb{P}(\mathrm{V}_{\smash{\mathrm{G}}}^{\lambda,\vee})$
attached to the lowest weight vector~$v_{-\lambda}$
of the~$\mathrm{G}$-representation~$\mathrm{V}_{\smash{\mathrm{G}}}^{\lambda,\vee}$.
The list of parabolic subgroups to consider
is obtained by considering the coefficients (in the basis of fundamental weights)
of the highest root of the root system of~$\mathrm{G}$.
The relevant geometric properties of adjoint partial flag varieties\footnote{
  We have omitted type~$\mathrm{C}_n$,
  because there is an exceptional isomorphism~$\SGr(1,2n)\cong\mathbb{P}^{2n-1}$,
  so we exclude it for the proof of \cref{theorem:main}.
}
are easily computed from the structure of the root system,
see, e.g., \cite[page~124]{MR1334091},
and are collected in \cref{table:adjoint},

\begin{table}[ht!]
  \centering
  \begin{tabular}{ccccc}
    \toprule
    type                            & variety                               & diagram                                   & dimension & index $\mathrm{i}_X$ \\
    \midrule
    $\mathrm{A}_n/\mathrm{P}_{1,n}$ & $\mathbb{P}(\tangent_{\mathbb{P}^n})$ & \dynkin[parabolic=9]{A}{}                 & $2n-1$    & $n$ \\
    $\mathrm{B}_n/\mathrm{P}_2$     & $\OGr(2,2n+1)$                        & \dynkin[parabolic=2]{B}{}                 & $4n-5$    & $2n-2$ \\
    $\mathrm{D}_n/\mathrm{P}_2$     & $\OGr(2,2n)$                          & \dynkin[parabolic=2]{D}{}                 & $4n-7$    & $2n-3$ \\
    $\mathrm{E}_6/\mathrm{P}_2$     &                                       & \dynkin[parabolic=2]{E}{6}                & $21$      & $11$ \\
    $\mathrm{E}_7/\mathrm{P}_1$     &                                       & \dynkin[parabolic=1]{E}{7}                & $33$      & $17$ \\
    $\mathrm{E}_8/\mathrm{P}_8$     &                                       & \dynkin[parabolic=128]{E}{8}              & $57$      & $29$ \\
    $\mathrm{F}_4/\mathrm{P}_1$     &                                       & \dynkin[parabolic=1]{F}{4}                & $15$      & $8$ \\
    $\mathrm{G}_2/\mathrm{P}_2$     & $\operatorname{G_2Gr}(2,7)$           & \dynkin[parabolic=2,reverse arrows]{G}{2} & $5$       & $3$ \\
    \bottomrule
  \end{tabular}
  \caption{Adjoint partial flag varieties}
  \label{table:adjoint}
\end{table}

To prove \cref{theorem:main},
we will use a filtration for the tangent bundle,
which is the dual of the lower central series for the nilradical~$\mathfrak{n}$
of the parabolic subalgebra~$\mathfrak{p}=\operatorname{Lie}\mathrm{P}$,
as~$\mathrm{T}_X$ is the equivariant vector bundle
induced by the~$\mathfrak{p}$-representation~$\mathfrak{g}/\mathfrak{p}$,
which is isomorphic to~$\mathfrak{n}^\vee$.
The origin and some properties of this filtration are described in \cite[\S3.1]{MR4706032}.
The main properties we will need are given by the following,
extending \cite[Lemma~39]{MR4706032}.
\begin{lemma}
  \label{lemma:tangent-bundle-adjoint}
  Let~$X=\GP$ be an adjoint partial flag variety,
  of dimension~$2r+1$.
  The tangent bundle $\mathrm{T}_X$ can be written as
  a non-split extension
  \begin{equation}
    \label{equation:T_X}
    0 \to
    \mathcal{E} \to
    \mathrm{T}_X \to
    \mathcal{O}_X(1) \to
    0
  \end{equation}
  of homogeneous vector bundles,
  where
  \begin{itemize}
    \item $\mathcal{E}$ is the irreducible $\mathrm{G}$-equivariant vector bundle
      of rank~$2r$ associated to the dual of the $\mathrm{L}$-representation
      $\mathfrak{n}^{\mathrm{ab}}=\mathfrak{n}/[\mathfrak{n},\mathfrak{n}]$,
      whose highest weight is given in \cref{table:E-weights},
      and
    \item $\mathcal{O}_X(1)$ is the ample line bundle
      attached to the dual of the
      one-dimensional~$\mathrm{L}$-representation~$[\mathfrak{n},\mathfrak{n}]$,
      which is also the ample line bundle
      for which~$\mathcal{O}_X(\mathrm{i}_X)\cong\omega_X^\vee$.
  \end{itemize}
\end{lemma}
We summarise its proof from \cite{MR4706032},
adding a direct description of the bundle~$\mathcal{E}$.
\begin{proof}
  By the weight computation in \cite[Lemma~28]{MR4706032}
  the nilradical~$\mathfrak{n}$ is a Heisenberg Lie algebra,
  which implies that its lower central series
  is of the form
  \begin{equation}
    \label{equation:lower-central-series}
    0\to
    [\mathfrak{n},\mathfrak{n}]\to
    \mathfrak{n}\to
    \mathfrak{n}/[\mathfrak{n},\mathfrak{n}]=\mathfrak{n}^{\mathrm{ab}}\to
    0,
  \end{equation}
  where~$\dim_\field[\mathfrak{n},\mathfrak{n}]=1$.
  The same weight computation shows that~$[\mathfrak{n},\mathfrak{n}]$
  is in fact the weight space~$\mathfrak{g}_{-\Theta}$,
  where~$\Theta$ is the highest root of~$\mathfrak{g}$.
  The sequence \eqref{equation:lower-central-series}
  induces after dualizing
  (as it is~$\mathfrak{n}^\vee=\mathfrak{g}/\mathfrak{p}$ which gives the tangent bundle)
  a sequence of the form \eqref{equation:T_X},
  with the line bundle associated to~$[\mathfrak{n},\mathfrak{n}]^\vee$
  corresponding to the projective realization of the adjoint variety,
  which is the generator~$\mathcal{O}_X(1)$ of the Picard group,
  as we have excluded type~C.

  To identify the bundle~$\mathcal{E}$,
  note that the action of~$\mathfrak{n}$ on~$\mathfrak{n}^{\mathrm{ab}}$ is trivial,
  so that the associated vector bundle~$\mathcal{E}$ is completely reducible.
  The highest weights which characterize~$\mathcal{E}$
  can be deduced from Kostant's theorem
  computing Lie algebra cohomology of the nilradical \cite[Corollary~8.2]{MR142696}.
  Note that \cite[Corollary~8.2]{MR142696} gives the lowest weights.

  We are interested
  in~$\HH_{\mathrm{CE}}^1(\mathfrak{n},\field)=\mathfrak{n}^{\mathrm{ab},\vee}$,
  which is given by a sum over coset representatives of length~1.
  By maximality of the parabolic subgroup,
  there is a \emph{unique} minimal length coset representative to consider
  (hence~$\mathcal{E}$ is irreducible),
  and it is the simple reflection~$\mathrm{s}_k$,
  so that the representation we are interested in has lowest weight
  \begin{equation}
    \mathrm{s}_k\cdot 0
    =
    \mathrm{s}_k(0+\rho)-\rho
    =
    \rho-\langle\rho,\alpha_k^\vee\rangle\alpha_k-\rho
    =
    -\alpha_k.
  \end{equation}
  Translating this to highest weights
  (or geometrically, dualizing the associated vector bundle, as in
  \cite[Lemma~2.1(1)]{2107.07814v3})
  we obtain
  \begin{equation}
    -w_0^{\mathrm{L}}(-\alpha_k)
    =
    -\alpha_k+\Theta.
  \end{equation}
  Finally, to express~$-\alpha_k+\Theta$ in terms of fundamental weights,
  one considers the~$k$th row of the Cartan matrix
  to know~$\alpha_k$,
  and an expression for the highest root~$\Theta$ in terms of fundamental weights
  can be looked up, see, e.g., \cite{MR240238},
  giving the highest weight in \cref{table:E-weights}.
\end{proof}

Consequently,
for~$q\geq 1$
we have the following short exact sequence
\begin{equation}
  \label{equation:exterior-power-T_X}
  0 \to
  \bigwedge^q \mathcal{E} \to
  \bigwedge^q \mathrm{T}_X \to
  \bigwedge^{q-1} \mathcal{E} \otimes\mathcal{O}_X(1) \to
  0
\end{equation}
of homogeneous vector bundles.
It is shown in \cite[Proposition~29]{MR4706032}
that for all~$q=0,\ldots,\dim X$
the bundle~$\bigwedge^q\mathrm{T}_X\cong\Omega_X^{\dim X-q}\otimes\omega_X^\vee$
has no higher cohomology,
which is why we have to consider different ample twists of~$\Omega_X^{\dim X-q}$
to show that Bott vanishing fails.

\begin{table}
  \centering
  \begin{tabular}{cc}
    \toprule
    type & $\mathcal{E}$ \\
    \midrule
    $\mathrm{B}_n$ & $\begin{cases} \mathcal{U}^{\omega_1-\omega_2+2\omega_3} & n=3 \\ \mathcal{U}^{\omega_1-\omega_2+\omega_3} & n\geq 4 \end{cases}$ \\
    $\mathrm{D}_n$ & $\begin{cases} \mathcal{U}^{\omega_1-\omega_2+\omega_3+\omega_4} & n=4 \\ \mathcal{U}^{\omega_1-\omega_2+\omega_3} & n\geq 5 \end{cases}$ \\
    $\mathrm{E}_6$ & $\mathcal{U}^{-\omega_2+\omega_4}$ \\
    $\mathrm{E}_7$ & $\mathcal{U}^{-\omega_1+\omega_3}$ \\
    $\mathrm{E}_8$ & $\mathcal{U}^{\omega_7-\omega_8}$ \\
    $\mathrm{F}_4$ & $\mathcal{U}^{-\omega_1+\omega_2}$ \\
    $\mathrm{G}_2$ & $\mathcal{U}^{3\omega_1-\omega_2}$ \\
    \bottomrule
  \end{tabular}
  \caption{$\mathcal{E}$ as an irreducible vector bundle}
  \label{table:E-weights}
\end{table}

We will repeatedly apply the following standard lemma.
\begin{lemma}
  \label{lemma:homogeneous-spectral-sequence}
  Let~$X=\GP$ be a partial flag variety.
  Let~$\mathcal{U}$ be a homogeneous vector bundle,
  and let~$\mathcal{F}^\bullet$ be the Jordan--H\"older filtration
  in the category of homogeneous vector bundles,
  with associated graded pieces~$\mathcal{G}^\bullet$,
  which are completely reducible homogeneous vector bundles.
  There is a spectral sequence
  \begin{equation}
    \mathrm{E}_1^{i,q-i}
    \colonequals
    \HH^q(X,\mathcal{G}^i)
    \Rightarrow
    \HH^q(X,\mathcal{U}),
  \end{equation}
  of~$\mathrm{G}$-representations.
  In particular,
  if a~$\mathrm{G}$-representation~$V$ appears in a single cohomological degree
  on the~$\mathrm{E}_1$-page
  then~$V$ also appears on the~$\mathrm{E}_\infty$-page.
\end{lemma}

To compute the cohomology of~$\mathcal{G}^i$
in \cref{lemma:homogeneous-spectral-sequence}
one uses the Borel--Weil--Bott theorem,
see, e.g., \cite[Theorem~5.0.1]{MR1038279}.
For the reader's sake
we state it in the shape we use for later computations.
Recall that a weight~$\lambda\in\mathrm{X}(\mathrm{T})_{\mathrm{G}}$
is regular (resp.~singular)
if for every root~$\alpha\in\mathrm{R}(\mathrm{G})$ we have~$(\lambda,\alpha)\neq0$
(resp.~there exists a root~$\alpha$ for which~$(\lambda,\alpha)=0$),
where~$(-,-)$ denotes the unique Weyl-invariant inner product
on~$\mathrm{X}(\mathrm{T})_{\mathrm{G}}$,
for~$\mathrm{T}\subset\mathrm{G}$ the fixed maximal torus.
\begin{theorem}[Borel--Weil--Bott]
  \label{theorem:borel-weil-bott}
  Let~$\mathrm{G}$ be a simply connected simple algebraic group,
  and let~$\mathrm{P}\subset\mathrm{G}$ be a parabolic subgroup.
  Let~$\mathcal{U}^\lambda$ be
  the~$\mathrm{G}$-equivariant vector bundle on~$\GP$,
  associated to the irreducible representation of the Levi subgroup~$\mathrm{L}\subset\mathrm{P}$
  with highest weight~$\lambda\in\mathrm{X}(\mathrm{T})_{\mathrm{L}}^+$.
  Then precisely one of the following holds:
  \begin{enumerate}
    \item if~$\lambda+\rho$ is a singular weight,
      then
      \begin{equation}
        \HH^\bullet(\GP,\mathcal{U}^\lambda)=0;
      \end{equation}
    \item if~$\lambda+\rho$ is a regular weight,
      then there exists a unique Weyl group element~$w\in\mathrm{W}_{\mathrm{G}}$
      such that~$w(\lambda+\rho)$
      is dominant,
      and
      \begin{equation}
        \HH^\bullet(\GP,\mathcal{U}^\lambda)
        \cong
        \mathrm{V}_{\mathrm{G}}^{w(\lambda+\rho)-\rho}[-\ell(w)].
      \end{equation}
  \end{enumerate}
\end{theorem}

\paragraph{Exterior powers}
The computation of the highest weights of exterior powers
of irreducible vector bundles
is a standard procedure,
which is for instance described in \cite[\S2.3]{2107.07814v3}.
We briefly recall this setup,
as we will need it for the computations.

The computation in the case of a maximal parabolic subgroup~$\mathrm{P}=\mathrm{P}_k$
decomposes into
\begin{itemize}
  \item a computation
    in the derived subgroup~$\mathrm{L}'$ of the Levi
    subgroup~$\mathrm{L}\subset\mathrm{P}$,
    which is simply connected semisimple
    and determined by omitting the vertex~$k$
    from the Dynkin diagram for~$\mathrm{G}$,
    and
  \item the one-dimensional torus~$\mathrm{Z}(\mathrm{L})$,
\end{itemize}
as there exists a surjective morphism of
algebraic groups~$\mathrm{L}'\times\mathrm{Z}(\mathrm{L})\to\mathrm{L}$.
On the level of weight lattices,
recall that
the fundamental weights of~$\mathrm{L}'$
are the fundamental weights of~$\mathrm{L}$
(and thus the fundamental weights of~$\mathrm{G}$)
excluding the fundamental weight~$\omega_k$ associated to the marked vertex,
and that there is an obvious lifting morphism in the opposite direction.
If~$\lambda$ is an element of the weight lattice of~$\mathrm{L}$,
we will write~$\lambda'$ for the induced weight of~$\mathrm{L}'$,
and vice versa,
using the lifting.

We also have a restriction morphism from the weight lattice of~$\mathrm{L}$
to that of~$\mathrm{Z}(\mathrm{L})$;
and working with rational coefficients,
the weight~$\lambda$ is sent to~$r_\lambda\colonequals(\lambda,\omega_k)/(\omega_k,\omega_k)$.
With this notation we have the following \cite[Lemma~2.4(2)]{2107.07814v3}.
\begin{lemma}
  \label{lemma:exterior-power}
  Let~$\mathrm{V}_{\mathrm{L}}^\lambda$
  be the irreducible~$\mathrm{L}$-representation
  of highest weight~$\lambda$.
  Let~$\mathrm{V}_{\mathrm{L}'}^{\lambda'}$
  be the induced irreducible representation of~$\mathrm{L}'$.
  If
  \begin{equation}
    \label{equation:exterior-power-restricted}
    \bigwedge^q\mathrm{V}_{\mathrm{L}'}^{\lambda'}
    =
    \bigoplus_{\nu'\in\Sigma}(\mathrm{V}_{\mathrm{L}'}^{\nu'})^{\oplus m(\lambda',\nu')}
  \end{equation}
  for some indexing set of weights~$\Sigma$
  and multiplicities~$m(\lambda',\nu')$,
  then
  \begin{equation}
    \label{equation:exterior-power-lifted}
    \bigwedge^q\mathrm{V}_{\mathrm{L}}^\lambda
    =
    \bigoplus_{\nu\in\Sigma}
    (\mathrm{V}_{\mathrm{L}}^{\nu+(qr_\lambda-r_\nu)\omega_k})^{\oplus m(\lambda',\nu')}.
  \end{equation}
  This implies that
  \begin{equation}
    \bigwedge^q\mathcal{U}^\lambda
    \cong
    \bigoplus_{\nu\in\Sigma}
    (\mathcal{U}^{\nu+(qr_\lambda-r_\nu)\omega_k})^{\oplus m(\lambda',\nu')}.
  \end{equation}
\end{lemma}
This computation can be implemented on a computer \cite{bott-non-vanishing}:
because~$\mathrm{L}'$ is semisimple,
the exterior power \eqref{equation:exterior-power-restricted}
can be computed using the Weyl character ring implementation in \cite{sagemath}.
The result in \eqref{equation:exterior-power-lifted}
is thus obtained by
lifting the semisimple computation for~$\mathrm{L}'$,
and adding~$(qr_\lambda-r_\nu)\omega_k$ to it.

In the exceptional types~$\mathrm{E}_{6,7,8}$, $\mathrm{F}_4$ and~$\mathrm{G}_2$,
computing the cohomology of the appropriate twisted exterior power of~$\mathcal{E}$
is all we have to do.

However, in type~$\mathrm{B}_n$ (resp.~$\mathrm{D}_n$)
we will need a computation
of the second and third exterior power of the vector bundle~$\mathcal{E}$
from \cref{table:E-weights},
which we require to be (almost) \emph{uniform} in~$n$,
so that we can treat all adjoint orthogonal Grassmannians simultaneously.
Observe that
the computation of the right-hand side in
\eqref{equation:exterior-power-restricted} is reduced,
via a plethysm,
to a computation in types~$\mathrm{A}_1$ and~$\mathrm{B}_{n-2}$ (resp.~$\mathrm{D}_{n-2}$),
because of the shape of the marked diagram in \cref{table:adjoint}.
From \cref{table:E-weights}
we see that the highest weight on the factor~$\mathrm{B}_{n-2}$ (resp.~$\mathrm{D}_{n-2}$)
is~$\omega_1'$
(where~$\omega_1',\ldots,\omega_{n-2}'$ are the fundamental weights of this factor).
This allows us to read off the computation of the components in the plethysm
for~$\mathrm{B}_{n-2}$ and~$\mathrm{D}_{n-2}$
from \cite[Table~5]{MR1064110}.

The main observation from the formulas in \cite[Table~5]{MR1064110}
is that, for~$n$ sufficiently large,
the only fundamental weights appearing
in the computation for type~$\mathrm{B}_{n-2}$ (resp.~$\mathrm{D}_{n-2}$)
will be~$\omega_1',\omega_2',\omega_3'$
(resp.~$\omega_1',\omega_2',\omega_3',\omega_4'$)
and the coefficients will be independent of~$n$
as soon as~$n\geq 6$ (resp.~$n\geq 7$),
because then the computation will only involve roots of the same length.
This allows us to consider only finitely many cases.
Likewise,
the computation of the coefficient of~$\omega_2$
in \cref{lemma:exterior-power},
the fundamental weight corresponding to the marked vertex in types~B and~D,
is also completely determined as soon as we have done
the computation for~$n=6$ (resp.~$n=7$).

We phrase this observation as the following lemma.
\begin{lemma}
  \label{lemma:only-first-coefficients-matter}
  Let~$\mathcal{U}^\lambda$
  be an irreducible summand of~$\bigwedge^2\mathcal{E}$ or~$\bigwedge^3\mathcal{E}$,
  where~$\mathcal{E}$ is the bundle from \cref{table:E-weights},
  in type~$\mathrm{B}_n$ (resp.~$\mathrm{D}_n$).
  Then the coefficients of~$\lambda$
  in terms of the fundamental weights
  are possibly non-zero only for~$\omega_1,\ldots,\omega_5$,
  and they are independent of~$n$ for~$n\geq 6$ (resp.~$n\geq 7$).
\end{lemma}

\paragraph{Dynkin type \texorpdfstring{$\mathrm{A}_n$}{A}}
The adjoint partial variety of type~$\mathrm{A}_n$,
or incidence variety,
is not a generalised Grassmannian,
but rather has Picard rank~2.
It is isomorphic to~$\mathbb{P}(\mathrm{T}_{\mathbb{P}^n})$,
and can thus be described as a degree-$(1,1)$ hypersurface
in~$\mathbb{P}^n\times\mathbb{P}^n$.
Using this description,
the following is already shown in \cite[\S4.2]{MR1464183}.
\begin{proposition}
  \label{proposition:adjoint-A}
  Let~$X$ be the adjoint partial flag variety of type~$\mathrm{A}_n$.
  Then
  \begin{equation}
    \HH^1(X,\tangent_X\otimes\mathcal{O}_X(-1))
    \cong\HH^1(X,\Omega_X^{2n-2}(n-1,n-1))
    \cong\field.
  \end{equation}
\end{proposition}

\paragraph{Dynkin type \texorpdfstring{$\mathrm{B}_n$}{B}}
We must assume that~$n\geq 3$:
when~$n=2$ there exists
the exceptional isomorphism~$\OGr(2,5)\cong\mathbb{P}^3$,
so we have to exclude it for the proof of \cref{theorem:main}.

\begin{proposition}
  \label{proposition:adjoint-B}
  Let~$X=\OGr(2,2n+1)$ be the adjoint partial flag variety of type~$\mathrm{B}_n$
  for~$n\geq 3$.
  Then
  \begin{equation}
    \begin{aligned}
      \HH^1(X,\bigwedge^3\mathrm{T}_X\otimes\mathcal{O}_X(-1))
      &\cong\HH^1(X,\Omega_X^{4n-8}(2n-3)) \\
      &\cong
      \begin{cases}
        \HH^1(X,\mathcal{U}^{\omega_1-2\omega_2+4\omega_3}) & n=3 \\
        \HH^1(X,\mathcal{U}^{\omega_1-2\omega_2+\omega_3+2\omega_4}) & n=4 \\
        \HH^1(X,\mathcal{U}^{\omega_1-2\omega_2+\omega_3+\omega_4}) & n\geq 5 \\
      \end{cases} \\
      &\cong
      \begin{cases}
        \mathrm{V}_\mathrm{G}^{2\omega_3} & n=3 \\
        \mathrm{V}_\mathrm{G}^{2\omega_4} & n=4 \\
        \mathrm{V}_\mathrm{G}^{\omega_4}  & n\geq 5.
      \end{cases}
    \end{aligned}
  \end{equation}
\end{proposition}

\begin{proof}
  The first isomorphism comes from the usual pairing on exterior powers of vector bundles,
  the dimension of~$X$,
  and the value of the index in \cref{table:adjoint}.

  The next isomorphism
  follows from the computation
  of the irreducible homogeneous vector bundles
  appearing in the outer terms of~\eqref{equation:exterior-power-T_X}
  for~$q=3$,
  and \cref{lemma:homogeneous-spectral-sequence}.
  When~$n=3,4,5,6$ the exterior power can be computed using \cref{lemma:exterior-power},
  as implemented in \cite{bott-non-vanishing}\footnote{
    This Git repository contains the (documented) code and its output.
  }.
  By \cref{lemma:only-first-coefficients-matter}
  it suffices to compute the case~$n=6$
  to know what happens for all~$n\geq 6$:
  we consider the third exterior power of~$\mathcal{E}$,
  which will at most involve coefficients of the fundamental weight~$\omega_5$.
  One observes that
  the non-zero cohomology is independent of~$n$ for~$n\geq 5$,
  but the intermediate computation is only independent of~$n$ for~$n\geq 6$.

  The final isomorphism is an application of the Borel--Weil--Bott theorem,
  which is again independent of~$n$ for~$n\geq 5$
  by the reasoning in \cref{lemma:only-first-coefficients-matter}.
  These computations are also implemented in \cite{bott-non-vanishing}.
\end{proof}

\begin{remark}
  \label{remark:interpretation-E-type-B}
  One can show that
  \begin{equation}
    \mathcal{E}\cong\mathcal{U}^\vee\otimes(\mathcal{Q}^\vee/\mathcal{U}),
  \end{equation}
  where~$\mathcal{U}$ is the universal subbundle,
  and~$\mathcal{Q}$ is the universal quotient bundle.
\end{remark}

\paragraph{Dynkin type \texorpdfstring{$\mathrm{D}_n$}{D}}
We can assume that~$n\geq 4$ (if it is not already excluded by convention):
when~$n=3$ there exists the exceptional
isomorphism~$\OGr(2,6)\cong\mathbb{P}(\mathrm{T}_{\mathbb{P}^3})$,
which is covered by \cref{proposition:adjoint-A}.

\begin{proposition}
  \label{proposition:adjoint-D}
  Let~$X=\OGr(2,2n)$ be the adjoint partial flag variety of type~$\mathrm{D}_n$
  for~$n\geq 4$.
  Then
  \begin{equation}
    \begin{aligned}
      \HH^1(X,\bigwedge^3\mathrm{T}_X(-1))
      &\cong\HH^1(X,\Omega_X^{4n-10}(2n-4)) \\
      &\cong
      \begin{cases}
        \HH^1(X,\mathcal{U}^{3\omega_1-2\omega_2+\omega_3+\omega_4}
          \oplus\mathcal{U}^{\omega_1-2\omega_2+3\omega_3+\omega_4}
        \oplus\mathcal{U}^{\omega_1-2\omega_2+\omega_3+3\omega_4})
        & n=4 \\
        \HH^1(X,\mathcal{U}^{\omega_1-2\omega_2+\omega_3+\omega_4+\omega_5})
        & n=5 \\
        \HH^1(X,\mathcal{U}^{\omega_1-2\omega_2+\omega_3+\omega_4})
        & n\geq 6
      \end{cases} \\
      &\cong
      \begin{cases}
        \mathrm{V}_\mathrm{G}^{2\omega_1} \oplus
        \mathrm{V}_\mathrm{G}^{2\omega_3} \oplus \mathrm{V}_\mathrm{G}^{2\omega_4} & n=4 \\
        \mathrm{V}_\mathrm{G}^{\omega_4+\omega_5} & n=5 \\
        \mathrm{V}_\mathrm{G}^{\omega_4} & n\geq6.
      \end{cases}
    \end{aligned}
  \end{equation}
\end{proposition}

The proof is parallel to that of \cref{proposition:adjoint-B}.

\begin{remark}
  \label{remark:interpretation-E-type-D}
  Similar to \cref{remark:interpretation-E-type-B},
  one can show that
  \begin{equation}
    \mathcal{E}\cong\mathcal{U}^\vee\otimes(\mathcal{Q}^\vee/\mathcal{U}),
  \end{equation}
  where~$\mathcal{U}$ is the universal subbundle,
  and~$\mathcal{Q}$ is the universal quotient bundle.
\end{remark}

\paragraph{Exceptional Dynkin types}
For the exceptional adjoint varieties
it turns out we need to consider an exterior power
which depends the type.
The following properties,
whose proof always consists of computing~$\bigwedge^q\mathcal{E}(-1)$
and~$\bigwedge^{q-1}\mathcal{E}$
for the indicated value of~$q$,
and verifying that the cohomology in degree~1
cannot be cancelled by the global sections,
similar to the proofs of \cref{proposition:adjoint-B,proposition:adjoint-D},
is implemented in the accompanying code and its output \cite{bott-non-vanishing}.
We describe the conclusions here.
\begin{proposition}
  \label{proposition:adjoint-E6}
  Let~$X$ be the adjoint partial flag variety in type~$\mathrm{E}_6$.
  Then we have
  \begin{equation}
    \HH^1(X,\bigwedge^5\mathrm{T}_X\otimes\mathcal{O}_X(-1))
    \cong\HH^1(X,\Omega_X^{16}(10))
    \cong\HH^1(X,\mathcal{U}^{2\omega_1-2\omega_2+\omega_4+2\omega_6})
    \cong\mathrm{V}_\mathrm{G}^{2\omega_1+2\omega_6}
  \end{equation}
\end{proposition}

\begin{proposition}
  \label{proposition:adjoint-E7}
  Let~$X$ be the adjoint partial flag variety in type~$\mathrm{E}_7$.
  Then we have
  \begin{equation}
    \HH^1(X,\bigwedge^7\mathrm{T}_X\otimes\mathcal{O}_X(-1))
    \cong\HH^1(X,\Omega_X^{26}(16))
    \cong\HH^1(X,\mathcal{U}^{-2\omega_1+\omega_3+3\omega_6})
    \cong\mathrm{V}_\mathrm{G}^{3\omega_6}.
  \end{equation}
\end{proposition}

\begin{proposition}
  \label{proposition:adjoint-E8}
  Let~$X$ be the adjoint partial flag variety in type~$\mathrm{E}_8$.
  Then we have
  \begin{equation}
    \HH^1(X,\bigwedge^{11}\tangent_X\otimes\mathcal{O}_X(-1))
    \cong\HH^1(X,\Omega_X^{46}(28))
    \cong\HH^1(X,\mathcal{U}^{5\omega_1+\omega_7-2\omega_8})
    \cong\mathrm{V}_{\mathrm{G}}^{5\omega_1}.
  \end{equation}
\end{proposition}

\begin{proposition}
  \label{proposition:adjoint-F4}
  Let~$X$ be the adjoint partial flag variety in type~$\mathrm{F}_4$.
  Then we have
  \begin{equation}
    \HH^1(X,\bigwedge^4\mathrm{T}_X\otimes\mathcal{O}_X(-1))
    \cong\HH^1(X,\Omega_X^{11}(7))
    \cong\HH^1(X,\mathcal{U}^{-2\omega_1+\omega_2+3\omega_4})
    \cong\mathrm{V}_\mathrm{G}^{3\omega_4}.
  \end{equation}
\end{proposition}

\begin{proposition}
  \label{proposition:adjoint-G2}
  Let~$X$ be the adjoint partial flag variety in type~$\mathrm{G}_2$.
  Then we have
  \begin{equation}
    \HH^1(X,\bigwedge^2\mathrm{T}_X\otimes\mathcal{O}_X(-1))
    \cong\HH^1(X,\Omega_X^4(2))
    \cong\HH^1(X,\mathcal{U}^{4\omega_1-2\omega_2})
    \cong\mathrm{V}_\mathrm{G}^{\omega_1}.
  \end{equation}
\end{proposition}

\section{Coadjoint case}
\label{section:coadjoint}
Now we turn to coadjoint partial flag varieties.
Their explicit geometric realization is obtained by
considering the unique closed~$\mathrm{G}$-orbit
using the~$\mathrm{G}$-dominant weight~$\lambda$ for the \emph{coadjoint} representation.

The distinction between adjoint and coadjoint is only meaningful
when the Dynkin diagram is not simply-laced,
so in types~$\mathrm{A}_n$, $\mathrm{D}_n$ and~$\mathrm{E}_{6,7,8}$
there is nothing left to do.
Moreover,
the coadjoint varieties in type~$\mathrm{B}_n$
(resp.~the coadjoint variety in type~$\mathrm{G}_2$)
are the odd-dimensional quadric hypersurfaces,
(resp.~the~5\dash dimensional quadric hypersurface),
and both varieties are also cominuscule of type~B.
These have already been considered in \cite[\S4.3]{MR1464183}.
This leaves us with the two cases in \cref{table:coadjoint}.

\begin{table}[ht!]
  \centering
  \begin{tabular}{cccccc}
    \toprule
    type                        & variety      & diagram                    & dimension & index $\mathrm{i}_X$ \\
    \midrule
    $\mathrm{C}_n/\mathrm{P}_2$ & $\SGr(2,2n)$ & \dynkin[parabolic=2]{C}{}  & $4n-5$    & $n+1$ \\
    $\mathrm{F}_4/\mathrm{P}_4$ &              & \dynkin[parabolic=8]{F}{4} & $15$      & $11$ \\
    \bottomrule
  \end{tabular}
  \caption{Coadjoint but not adjoint partial flag varieties}
  \label{table:coadjoint}
\end{table}

\paragraph{Dynkin type \texorpdfstring{$\mathrm{C}_n$}{C}}
We can assume that~$n\geq 3$:
when~$n=2$ there exists the exceptional isomorphism~$\SGr(2,6)\cong\mathrm{Q}^3$,
for which \cref{theorem:main} is already established by \cite[\S4.1]{MR1464183},
so we exclude it for the proof of \cref{theorem:main}.
\begin{proposition}
  \label{proposition:coadjoint-C}
  Let~$X=\SGr(2,2n)$ be the coadjoint partial flag variety of type~$\mathrm{C}_n$,
  where~$n\geq 3$.
  Then
  \begin{equation}
    \HH^1(X,\mathrm{T}_X\otimes\mathcal{O}_X(-1))
    \cong\HH^1(X,\Omega_X^{4n-6}(n))
    \cong\field.
  \end{equation}
\end{proposition}

\begin{proof}
  An analysis of the nilradical~$\mathfrak{n}$,
  similar to the proof of \cref{lemma:tangent-bundle-adjoint},
  using \cite[Lemma~33]{MR4706032},
  shows that~$\tangent_X\otimes\mathcal{O}_X(-1)$ is the unique non-split extension
  \begin{equation}
    0
    \to\mathcal{U}^{\omega_1-2\omega_2+\omega_3}
    \to\tangent_X\otimes\mathcal{O}_X(-1)
    \to\mathcal{U}^{2\omega_1-\omega_2}\to
    0,
  \end{equation}
  where~$\mathcal{U}^{2\omega_1-\omega_2}$ is a rank-3 bundle.
  As in the proof of \cref{proposition:adjoint-B,proposition:adjoint-D},
  the Borel--Weil--Bott computation for the outer terms in the sequence
  is uniform in all~$n$,
  and the result can be deduced
  either via a direct computation for~$n=3$,
  of from the implementation in~\cite{bott-non-vanishing}.
  We obtain that~$\HH^1(X,\mathcal{U}^{\omega_1-2\omega_2+\omega_3})\cong\field$
  and~$\HH^\bullet(X,\mathcal{U}^{2\omega_1-\omega_2})=0$,
  thus proving the claim.
\end{proof}

\paragraph{Exceptional Dynkin type \texorpdfstring{$\mathrm{F}_4$}{F}}
We end with the following exceptional case.
\begin{proposition}
  \label{proposition:coadjoint-F4}
  Let~$X$ be the coadjoint partial flag variety of type~$\mathrm{F}_4$.
  Then
  \begin{equation}
    \HH^1(X,\mathrm{T}_X\otimes\mathcal{O}_X(-1))
    \cong\HH^1(X,\Omega_X^{14}(10))
    \cong\field.
  \end{equation}
\end{proposition}

\begin{proof}
  An analysis of the nilradical~$\mathfrak{n}$,
  similar to the proof of \cref{lemma:tangent-bundle-adjoint},
  using \cite[Lemma~34]{MR4706032},
  shows that~$\mathrm{T}_X\otimes\mathcal{O}_X(-1)$ is the unique non-split extension
  \begin{equation}
    0\to\mathcal{U}^{\omega_3-2\omega_4}\to\tangent_X\otimes\mathcal{O}_X(-1)\to\mathcal{U}^{\omega_1-\omega_4}\to 0,
  \end{equation}
  where~$\mathcal{U}^{\omega_1-\omega_4}$ is a rank-7 bundle.
  Applying the Borel--Weil--Bott theorem to the outer terms in this sequence,
  cf.~\cite{bott-non-vanishing},
  we obtain that~$\HH^1(X,\mathcal{U}^{\omega_3-2\omega_4})\cong\field$
  and~$\HH^\bullet(X,\mathcal{U}^{\omega_1-\omega_4})=0$,
  thus proving the claim.
\end{proof}

\paragraph{Acknowledgements}
The author was supported by
NWO (Dutch Research Council)
as part of the grant
\href{https://doi.org/10.61686/RZKLF82806}{\texttt{doi:10.61686/RZKLF82806}}.
We would like to thank Max Briest and Maxim Smirnov for interesting discussions,
and the referee for comments that helped improve the paper.

\renewcommand*{\bibfont}{\normalfont\small}
\printbibliography

\emph{Pieter Belmans}, \url{p.belmans@uu.nl} \\
Mathematical Institute, Utrecht University, Budapestlaan 6, 3584 CD Utrecht, The Netherlands

\end{document}